\documentclass{amsproc}

\usepackage{amsmath}
\usepackage{amssymb}
\usepackage[mathcal]{eucal}
\usepackage[bb=pazo]{mathalfa}
\usepackage{rotating}
\usepackage{graphicx}
\usepackage[tableposition=below]{caption}
\usepackage{pigpen}
\usepackage{multirow}
\usepackage{bm}

\DeclareGraphicsExtensions{.png,.pdf,.eps}
\usepackage[all,2cell,cmtip]{xy}
\usepackage{tikz}
\usetikzlibrary{patterns,snakes}
\usepackage{color}
\usepackage{longtable}
\usepackage{soul}
\newtheorem{theorem}{Theorem}[section]
\newtheorem{lemma}[theorem]{Lemma}
\newtheorem{corollary}[theorem]{Corollary}
\newtheorem{proposition}[theorem]{Proposition}

\newtheorem{notation}[theorem]{Notation}

\theoremstyle{definition}

\newtheorem{remark}[theorem]{Remark}

\def\G{{\mathbb G}}

\def\P{{\mathbb P}}

\def\Z{{\mathbb Z}}

\def\cA{{\mathcal A}}
\def\cB{{\mathcal B}}
\def\cC{{\mathcal C}}
\def\cD{{\mathcal D}}
\def\cE{{\mathcal E}}
\def\cF{{\mathcal F}}
\def\cG{{\mathcal{G}}}
\def\cH{{\mathcal{H}}}

\def\cK{{\mathcal K}}
\def\cL{{\mathcal L}}
\def\cM{{\mathcal M}}
\def\cN{{\mathcal{N}}}
\def\cO{{\mathcal{O}}}

\def\cQ{{\mathcal{Q}}}

\def\cU{{\mathcal U}}
\def\cV{{\mathcal V}}
\def\cW{{\mathcal W}}

\def\cZ{{\mathcal Z}}

\def\G{{\mathbb{G}}}

\def\DA{{\rm A}}

\def\DC{{\rm C}}

\def\DF{{\rm F}}

\def\GL{\operatorname{GL}}
\def\PGL{\operatorname{PGL}}

\def\SP{\operatorname{Sp}}

\def\HU{{\rm H}\cU}
\def\HM{{\rm H}\cM}
\def\HG{{\rm H}\cG}

\def\lra{\longrightarrow}

\def\ra{\rightarrow}
\def\lra{\longrightarrow}
\def\rat{\dashrightarrow}

\def\operatorname#1{\mathop{\rm #1}\nolimits}

\def\Aut{\operatorname{Aut}}

\def\Chow{\operatorname{Chow}}

\def\Pic{\operatorname{Pic}}

\def\deg{\operatorname{deg}}

\def\det{\operatorname{det}}

\def\loc{\operatorname{Locus}}
\def\rat{\operatorname{RatCurves^n}}

\def\NE{{\operatorname{NE}}}

\newcommand{\ol}[1]{\overline{#1}}
\newcommand{\pb}{\ar@{}[dr]|(.25){\text{\pigpenfont J}}}

\makeatletter
 
\newcommand*\wthelper[2]{%
        \hbox{\dimen@\accentfontxheight#1%
                \accentfontxheight#11.15\dimen@
                $\m@th#1\widetilde{#2}$%
                \accentfontxheight#1\dimen@
        }%
}

\newcommand*\accentfontxheight[1]{%
        \fontdimen5\ifx#1\displaystyle
                \textfont
        \else\ifx#1\textstyle
                \textfont
        \else\ifx#1\scriptstyle
                \scriptfont
        \else
                \scriptscriptfont
        \fi\fi\fi3
}
\makeatother

\newcommand{\shse}[3]{0 ~\ra ~#1~ \lra ~#2~ \lra ~#3~ \ra~ 0}

\AtBeginDvi{}

\makeindex

\begin{document}

\title[A Characterization of Symplectic Grassmannians]{A Characterization of Symplectic Grassmannians}

\author[Occhetta]{Gianluca Occhetta}
\address{Dipartimento di Matematica, Universit\`a di Trento, via
Sommarive 14 I-38123 Povo di Trento (TN), Italy} 
\email{gianluca.occhetta@unitn.it}
\thanks{First author supported by PRIN project ``Geometria delle variet\`a algebriche'' Grant number 2010S47ARA\_010. First and second author supported by the Department of Mathematics of the University of Trento.}
\author[Sol\'a Conde]{Luis E. Sol\'a Conde}
\address{Dipartimento di Matematica, Universit\`a di Trento, via
Sommarive 14 I-38123 Povo di Trento (TN), Italy}
\email{lesolac@gmail.com}
\author[Watanabe]{Kiwamu Watanabe}
\address{Course of Mathematics, Programs in Mathematics, Electronics and Informatics,
Graduate School of Science and Engineering, Saitama University.
Shimo-Okubo 255, Sakura-ku Saitama-shi, 338-8570 Japan}
\email{kwatanab@rimath.saitama-u.ac.jp}
\thanks{The third author was partially supported by JSPS KAKENHI Grant Number 26800002.}
\begin{abstract}
We provide a characterization of Symplectic Grassmannians in terms of their Varieties of Minimal Rational Tangents.
\end{abstract}

\maketitle

\section{Introduction}\label{sec:intro}

Families of rational curves have been shown to be a powerful tool in birational geometry (see \cite{De,kollar}). Remarkably, the local analysis of these families via certain differential geometric techniques has been very useful in order to recover global properties of Fano manifolds of Picard number one (see \cite{Hw,Hw15}, and the references therein). In fact, many geometric features of these varieties can be read from the projective geometry of the set of tangent directions of a family of minimal rational curves at the general point of the variety --the so-called variety of minimal rational tangents, or VMRT, for short.

An archetypal result of this type, due to Mok and to Hong-Hwang \cite{Mk3,HH}, allows us to recognize a homogeneous Fano manifold $X$ of Picard number one by looking at its VMRT at a general point $x$, and at its embedding into $\P(\Omega_{X,x})$. In a nutshell, their strategy consists of constructing a biholomorphism between analytic open subsets of  $X$ and the model manifold preserving the 
VMRT structure, and then using the so-called Cartan-Fubini extension principle to extend it along the rational curves of the family. This characterization works 
for all rational homogeneous manifolds of Picard number one whenever the VMRT is rational homogeneous, which is always the case except for the short root cases; namely for symplectic Grassmannians, and for two varieties of type $\DF_4$. In the former cases, quoting Mok (\cite{Mk4}): ``{\it An
important feature for the primary examples of symplectic Grassmannians} [\dots\hspace{-0.06cm}] {\it is the
existence of local differential-geometric invariants which cannot possibly be captured by
the VMRT at a general point.}''

In this paper we show that, if we impose that the VMRT is the expected one at every point of the variety, we may still characterize symplectic Grassmannians (see Theorem \ref{thm:main} for the precise statement). In fact, we may use a family of minimal rational curves to construct a smooth projective variety $\ol{X}$, dominating $X$, supporting as many independent $\P^1$-bundle structures as its Picard number. Then we may use \cite{OSWW} to claim that $\ol{X}$ is a flag manifold (of type $\DC$), and $X$ is, subsequently, a symplectic Grassmannian. The same argument had been  successfully used in \cite{OSW} 
to characterize rational homogeneous manifolds of Picard number one corresponding to long roots. In that case the existence of a family of rational curves with smooth evaluation and the appropriate fibers allowed to complete the proof, without making any assumption on its image into $\P(\Omega_X)$. In the case of symplectic Grassmannians, the existence of an adequate embedding of the family into $\P(\Omega_X)$ seems to be crucial in order to carry out the proof.
We believe that a similar technique may lead also to the characterization of the last two remaining short root cases.

The structure of the paper is the following: Section \ref{sec:prelim} contains some preliminary material regarding rational curves, symplectic Grassmannians and flag bundles. In Section \ref{sec:VMRTC} we show how to reconstruct, upon the VMRT,  certain  bundles over the manifold $X$ (corresponding, a posteriori, to universal bundles on the symplectic Grassmannian), and study some of their properties. They can be used (see Section \ref{sec:isotrop}) to prove that a certain subvariety $\HU$ of the family of rational curves is indeed a subfamily. Section \ref{sec:reduc} contains the last technical ingredient of the proof: the existence of a nondegenerate skew-symmetric form on one of those bundles. This will allow us in Section \ref{sec:flags} to construct a flag bundle over $\HU$ of the appropriate type, which we prove to be a flag manifold.

\section{Preliminaries}\label{sec:prelim}

\subsection{Families of rational curves}

For the readers' convenience, we will introduce here some notation that we will use along the paper, regarding rational curves on algebraic varieties. We refer to \cite{kollar} for a complete account on this topic.
 
Given a smooth complex projective variety $Y$, a {\it family of rational curves on} $Y$ is, by definition, the normalization $\cN$ of an irreducible component of the scheme $\rat(Y)$. Each of these families comes equipped with a smooth $\P^1$-fibration $p:\cV\to \cN$ and an evaluation morphism $q:\cV\to Y$. 
The family $\cN$ is called {\it dominating} if $q$ is dominant, and {\it unsplit} if $\cN$ is proper. Given a point $y\in Y$, we denote by $\cN_y$ the normalization of $q^ {-1}(y)$, and by $\cV_y$ its fiber product with $\cV$ over $\cN$. For  a general point $y \in Y$, $\dim \cN_y =-K_Y \cdot \ell -2$, where $\ell$ is an element of the family. 
The composition $q^*\Omega_Y\to \Omega_{\cV}\to\Omega _{\cV/\cN}$ gives rise to a rational map $\tau:\cV\dashrightarrow \P(\Omega_Y)$, that we call the {\it tangent map}. A point $P\in\cV$ for which this map is not defined is called a {\it cusp} of the family.  

\begin{remark}\label{rem:O1} Let us denote by $Y^\circ\subset Y$ a nonempty subset such that the tangent map $\tau$ is defined at $\cV^\circ:=q^ {-1}(Y^\circ)$.
By definition, it satisfies
\begin{equation*}
\tau^{\ast}\cO_{\P(\Omega_{Y^\circ})}(1)=(\Omega_{\cV/\cN})_{|\cV^\circ}=\cO_{\cV^\circ}(K_{\cV/\cN}).
\end{equation*}
\end{remark}

We will be mostly interested in the case in which $\cN$ is a dominating, unsplit family such that $q:\cV\to Y$
is a smooth morphism, and we will simply say that the family is {\it beautiful}. Note that in this case the varieties $\cV$ and $\cN$ are smooth as well. On the other hand, the first two assumptions allow us to say, by \cite[Theorem~3.3]{Ke2}, that the set of cusps of $\cN$ does not dominate $Y$, and $\tau$ is defined over $\cN_y$, for $y\in Y$ general. We denote the restriction of $\tau$ to $\cN_y$ by $\tau_y$. Its image is usually called the {\it VMRT} of $\cN$ at $y$; it is known that, for a general $y$, $\tau_y$ is its normalization (see \cite[Theorem 1]{HM2}, \cite[Theorem 3.4]{Ke2}).

\begin{remark} Beautiful families arise naturally, for instance, in the context of Campana-Peternell conjecture:
any family of rational curves of minimal degree with respect to a fixed ample line bundle on a Fano manifold  with nef tangent bundle is a beautiful family (see for instance \cite[Proposition 2.10]{MOSWW}).
\end{remark}

\begin{notation}
Along the paper we will often consider vector bundles on the projective line $\P^1$. For simplicity, we will denote by $E(a_1^{k_1},\dots,a_r^{k_r})$ the vector bundle $\bigoplus_{j=1}^r\cO(a_j)^{\oplus k_j}$ on $\P^1$.
\end{notation}

\begin{remark}\label{rem:beautifulfree}
Given a curve $C=p^{-1}(c)$, $c\in\cN$, of a beautiful family $\cN$ of rational curves 
in $Y$, we have a surjective map $(T_{\cV})_{|C}\to (q^*T_{Y})_{|C}$. Since $(T_{\cV})_{|C}\cong E(2,0^{\dim(\cN)})$, it follows that $(q^*T_{Y})_{|C}$ is nef, that is, $q_{|C}:C\to Y$ is {\it free} (cf. \cite[Definition~II.3.1]{kollar}). In particular, the natural map from $\cN$ to $\Chow(Y)$ is injective, that is two different elements of $\cN$ cannot correspond to the same curve in $Y$ (see, for instance, \cite[Proof of Lemma~5.2]{KK}).
\end{remark}

\subsection{Statement of the Main Theorem}\label{ssec:statement}
 
The goal of this paper is to characterize rational homogeneous manifolds of the form $\SP(2n)/P_r$, where $P_r$ is a maximal parabolic subgroup, determined by a node $r\in\{2,\dots,n-1\}$ in the Dynkin diagram of $\SP(2n)$ --where the nodes of the diagram are numbered as in (\ref{eq:dynkins}). A variety of this kind is classically interpreted as the {\it Symplectic Grassmannian} parametrizing the $(r-1)$-dimensional linear subspaces in $\P^{2n-1}$ that are isotropic with respect to a certain nondegenerate skew-symmetric form. 
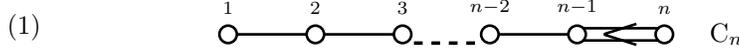
\begin{equation}\label{eq:dynkins}
\begin{array}{c}\vspace{-0.1cm}
\ifx\du\undefined
  \newlength{\du}
\fi
\setlength{\du}{3.3\unitlength}
\begin{tikzpicture}
\pgftransformxscale{1.000000}
\pgftransformyscale{1.000000}

\definecolor{dialinecolor}{rgb}{0.000000, 0.000000, 0.000000} 
\pgfsetstrokecolor{dialinecolor}
\definecolor{dialinecolor}{rgb}{0.000000, 0.000000, 0.000000} 
\pgfsetfillcolor{dialinecolor}


\pgfsetlinewidth{0.300000\du}
\pgfsetdash{}{0pt}
\pgfsetdash{}{0pt}

\pgfpathellipse{\pgfpoint{-6\du}{0\du}}{\pgfpoint{1\du}{0\du}}{\pgfpoint{0\du}{1\du}}
\pgfusepath{stroke}
\node at (-6\du,0\du){};

\pgfpathellipse{\pgfpoint{4\du}{0\du}}{\pgfpoint{1\du}{0\du}}{\pgfpoint{0\du}{1\du}}
\pgfusepath{stroke}
\node at (4\du,0\du){};

\pgfpathellipse{\pgfpoint{14\du}{0\du}}{\pgfpoint{1\du}{0\du}}{\pgfpoint{0\du}{1\du}}
\pgfusepath{stroke}
\node at (14\du,0\du){};

\pgfpathellipse{\pgfpoint{24\du}{0\du}}{\pgfpoint{1\du}{0\du}}{\pgfpoint{0\du}{1\du}}
\pgfusepath{stroke}
\node at (24\du,0\du){};

\pgfpathellipse{\pgfpoint{34\du}{0\du}}{\pgfpoint{1\du}{0\du}}{\pgfpoint{0\du}{1\du}}
\pgfusepath{stroke}
\node at (34\du,0\du){};

\pgfpathellipse{\pgfpoint{44\du}{0\du}}{\pgfpoint{1\du}{0\du}}{\pgfpoint{0\du}{1\du}}
\pgfusepath{stroke}
\node at (44\du,0\du){};

\pgfsetlinewidth{0.300000\du}
\pgfsetdash{}{0pt}
\pgfsetdash{}{0pt}
\pgfsetbuttcap

{\draw (-5\du,0\du)--(3\du,0\du);}
{\draw (5\du,0\du)--(13\du,0\du);}
{\draw (25\du,0\du)--(33\du,0\du);}
{\draw (34.65\du,0.7\du)--(43.35\du,0.7\du);}
{\draw (34.65\du,-0.7\du)--(43.35\du,-0.7\du);}


{\pgfsetcornersarced{\pgfpoint{0.300000\du}{0.300000\du}}\definecolor{dialinecolor}{rgb}{0.000000, 0.000000, 0.000000}
\pgfsetstrokecolor{dialinecolor}
\draw (40.8\du,-1.2\du)--(37\du,0\du)--(40.8\du,1.2\du);}

\pgfsetlinewidth{0.400000\du}
\pgfsetdash{{1.000000\du}{1.000000\du}}{0\du}
\pgfsetdash{{1.000000\du}{1.00000\du}}{0\du}
\pgfsetbuttcap
{\draw (15.3\du,-1\du)--(23\du,-1\du);}

\node[anchor=west] at (48\du,0\du){${\rm C}_n$};

\node[anchor=south] at (-6\du,1.1\du){$\scriptstyle 1$};

\node[anchor=south] at (4\du,1.1\du){$\scriptstyle 2$};

\node[anchor=south] at (14\du,1.1\du){$\scriptstyle 3$};

\node[anchor=south] at (24\du,1.1\du){$\scriptstyle n-2$};

\node[anchor=south] at (34\du,1.1\du){$\scriptstyle n-1$};

\node[anchor=south] at (44\du,1.1\du){$\scriptstyle n$};

\end{tikzpicture} \vspace{0.1cm}\\
\end{array}
\end{equation}

The Pl\"ucker embedding of $\SP(2n)/P_r$ is covered by lines, and the family of these lines in $\SP(2n)/P_r$ is beautiful. Furthermore, its evaluation morphism is isotrivial, with fibers isomorphic to the smooth projective variety 
$$\cC:=\P\big(\cO_{\P^{r-1}}(2)\oplus \cO_{\P^{r-1}}(1)^{2n-2r}\big), \qquad \quad r\in\{2,\dots,n-1\},$$ 
 and its tangent map is an embedding (see \cite[1.4.7]{Hw}). The restriction of this embedding to any fiber is projectively equivalent to the embedding of $\cC$ into $$\P\big(H^0(\cC,\cO_{\cC}(1))\big)=\P\big(H^0(\P^{r-1}, \cO_{\P^{r-1}}(2)\oplus \cO_{\P^{r-1}}(1)^{2n-2r})\big).$$

Note that the homogeneous 
manifold $\SP(2n)/P_{r-1,r+1}$ (where $P_{r-1,r+1}$ denotes the parabolic subgroup determined by the nodes $r-1$, $r+1$) can be identified with a subfamily of the family of lines in $\SP(2n)/P_r$, which we call the family of {\it isotropic lines}.
Its evaluation morphism has fibers isomorphic to $${\rm H}\cC:=\P\big(\cO_{\P^{r-1}}(1)^{2n-2r}\big).$$

We can now give a precise statement of the main result in this paper.

\begin{theorem}\label{thm:main}
Let $X$ be a Fano manifold of Picard number one, endowed with a beautiful family of rational curves. Assume that its tangent map $\tau$ is a morphism, satisfying that $\tau_x$ is projectively equivalent to the embedding of $\cC$ into $\P\big(H^0(\cC,\cO_{\cC}(1))\big)$, for every $x\in X$. Then $X$ is isomorphic to a symplectic Grassmannian.
\end{theorem}

\subsection{Flag bundles}\label{ssec:flag}

Let $G$ be a semisimple group, $P\subset G$ a parabolic subgroup, and $\cD$ the Dynkin diagram of $G$.
A $G/P$-bundle over a smooth projective variety $Y$ is a smooth projective morphism $q:E\to Y$
such that all the fibers of $q$ over closed points are isomorphic to $G/P$. If $P=B$ is a Borel subgroup of $G$ we will also call $E$ a flag bundle of type $\cD$ or a $\cD$-bundle. We present here some general facts on these bundles, without proofs, and refer the reader to \cite{OSW} for details.

Assume $Y$ is simply connected (for instance, if it is rationally connected); whenever the identity component of the automorphism group of $G/P$ is a semisimple group with the same Lie algebra as $G$, a $G/P$-bundle $q:E\to Y$ determines a $\cD$-bundle $\ol{q}:\ol{E}\to Y$ dominating it. Furthermore, in this case, given a parabolic subgroup $P'\subset G$, containing $B$, there exist a $G/P'$-bundle $q':E'\to Y$, and a contraction $\ol{q}':\ol{E}\to E'$ such that $\ol{q}=q'\circ \ol{q}'$, whose fibers are isomorphic to $P'/B$ (which is a rational homogeneous space, quotient of a Levi part of the group $P'$); for $P'=P$, we get our original bundle $q$. In particular, when $P'$ is a minimal parabolic subgroup properly containing $B$, then $\ol{q}':\ol{E}\to E'$ is a smooth $\P^1$-bundle. Given two parabolic subgroups $P'\subset P''$, containing $B$, the associated bundles and contractions fit in the following commutative diagram:
$$
\xymatrix{\ol{E}\ar[r]_{\ol{q}'}\ar[rd]_(.40){\ol{q}}\ar@/^{5mm}/[rr]^{\ol{q}''}&E'\ar[r]\ar[d]^(.40){q'}&E''\ar[ld]^(.40){q''}\\
&Y&}
$$

In the case $Y \simeq \P^1$, a $\cD$-bundle is determined by the intersection numbers of a minimal section $C$ of the bundle with the relative canonical bundles $K_i$ of the $\P^1$-bundle arising from the minimal parabolic subgroups $P_i\subset G$ properly containing $B$  (\cite[Proposition 3.17]{OSW}). Each of these subgroups corresponds to a node $i$ of the Dynkin diagram $\cD$, therefore we may represent the $\cD$-bundle by the {\it tagged Dynkin diagram} obtained by tagging the node $i$ with the intersection number $K_i\cdot C$ (see \cite[Remark 3.18]{OSW}).


\section{Projective geometry}\label{sec:VMRTC}

\subsection{Setup}\label{set:setup}

Throughout the rest of the paper $X$ will denote a Fano manifold of Picard number one, admitting a beautiful family of rational curves $\cM$,  satisfying the assumptions of Theorem \ref{thm:main}.
We denote by $p:\cU \to \cM$ the universal family, by
 $q:\cU \to X$ the evaluation morphism and by $\tau:\cU \to \P(\Omega_X)$ the tangent map; finally, given    $x \in X$, we denote by $\cM_x$ the fiber $q^{-1}(x)$.

\subsection{Projective geometry of the VMRT}
For every $x\in X$, the variety $\cM_x$ is isomorphic to the blow-up $\pi_{2,x}$ of a $(2n-r-1)$-dimensional projective space $\P(\cF_{2,x})$ along a $(2n-2r-1)$-dimensional subspace $\P(\cQ_x)$. Denoting by $\varphi_x$ the associated linear map $\cF_{2,x}\to\cQ_x$, and setting $\cF_{1,x}:=\ker(\varphi_x)$,  the exceptional divisor of the blow-up map $\cM_x\to\P(\cF_{2,x})$ is isomorphic to $\P(\cF_{1,x})\times\P(\cQ_x)$. Moreover, the blow-up can be seen as a resolution of indeterminacies $\pi_{1,x}:\cM_x\to \P(\cF_{1,x})$ of the linear projection from $\P(\cF_{2,x})$ onto $\P(\cF_{1,x})$. Summing up:
\begin{equation}\label{eq:projgeomC}
\xymatrix{&&\P(\cF_{1,x})\\ 
\P(\cF_{1,x})\times \P(\cQ_x)\,\,\ar@{^{(}->}[r] \ar@/^{4mm}/[rru]^*+<0.5em>{\mbox{\tiny nat. proj.}}\ar[rd]_*+<0.5em>{\mbox{\tiny nat. proj.}}&\cM_x\ar[ru]_{\,\P^{2n-2r}-\mbox{\tiny bdl.}}^{\pi_{1,x}}\ar[rd]^{\mbox{\tiny blow-up}}_{\pi_{2,x}}&\\ &\P(\cQ_x)\,\,\ar@{^{(}->}[r]&\P(\cF_{2,x})\ar@/_{8mm}/@{-->}[uu]_{\mbox{\tiny lin. proj.}}}
\end{equation}

The variety $\cM_x$ is embedded into $\P(\Omega_{X,x})$ via the complete linear system of the line bundle $\pi_{1,x}^*\cO_{\P(\cF_{1,x})}(1)\otimes \pi_{2,x}^*\cO_{\P(\cF_{2,x})}(1)$, whose restriction to $\P(\cF_{1,x})\times \P(\cQ_x)$ provides the Segre embedding of this variety into its linear span, which is  isomorphic to $\P(\cF_{1,x}\otimes\cQ_{x})$. We then have an epimorphism $\Omega_{X,x}\to \cF_{1,x}\otimes\cQ_{x}$ whose kernel is isomorphic to $S^2\cF_{1,x}$, since the linear projection of $\cM_x$ from $\P(\cF_{1,x}\otimes\cQ_{x})$ onto $\P(S^2\cF_{1,x})$ has  the second Veronese embedding of $\P(\cF_{1,x})$ as its image. 

\subsection{Relative projective geometry}\label{ssec:VMRTCglobal}

We will now see how the above constructions can be done relatively over $X$, under the assumptions of Setup \ref{set:setup}.

We start by noting that the automorphism group of $\cM_x$ equals the parabolic subgroup of $\Aut(\P(\cF_{2,x}))$ defined as the stabilizer $\Aut(\P(\cF_{2,x}))_{\P(\cQ_x)}$ of $\P(\cQ_x)$ (that is, the subgroup of automorphisms that leave $\P(\cQ_x)$ invariant). Since this holds for every $x\in X$, the theorem of Fischer and Grauert (cf. \cite[Theorem~I.10.1]{BPVV}) tells us that the variety $\cU$ is determined by a degree one cocycle $\theta$ on $X$ (considered as a complex manifold with the analytic topology) with values in the group $G:=\Aut(\P(\cF_{2,x}))_{\P(\cQ_x)}$. 

This group acts on the varieties $\P(\cQ_x)$  and $\P(\cF_{2,x})$ so that the inclusion $\P(\cQ_x)\hookrightarrow\P(\cF_{2,x})$ is equivariant and the action descends via the corresponding linear  projection to an action on $\P(\cF_{1,x})$, Then we conclude that Diagram $(\ref{eq:projgeomC})$ is a diagram of $G$-modules and equivariant morphisms. Then, denoting by $\cE\to X$ the $G$-principal bundle defined by the cocycle $\theta$, we may construct fiber bundles by means of $\cE$ and the action of $G$ on the above varieties; for instance, $\cE\times^G\cM_x$ is just the variety $\cU$. Then, defining 
$$
\begin{array}{l}
\cU_i:=\cE\times^G\P(\cF_{i,x}),\,\,\,i=1,2,\\
\cZ:=\cE\times^G\P(\cQ_{x}),\\
\HU:=\cU_1\times_X\cZ=\cE\times^G(\P(\cF_{1,x})\times\P(\cQ_{x})).
\end{array}
$$
we have a commutative diagram of fiber bundles on $X$:
$$
\xymatrix{&&\cU_1\ar[rd]^{q_1}&\\ \HU\,\ar@{^{(}->}[r]\ar@<0.5 ex>[rru]+<-2ex,0ex>\ar[rd]&\cU\ar[ru]^{\pi_1}\ar[rr]^q\ar[rd]_{\pi_2}&&X\\ &\cZ\,\ar@{^{(}->}[r]&\cU_2\ar[ru]_{q_2}&}
$$
 
On the other hand, we may consider the embedding $\cU\hookrightarrow \P(\Omega_X)$. Since every element of $G$ extends to a unique automorphism of $\P(\Omega_{X,x})$ (because $\cM_x$ is linearly normal and nondegenerate in this space) we may identify this embedding  with $\cE\times^G\cM_x\hookrightarrow\cE\times^G\P(\Omega_{X,x})$. 
Note that $\pi_1:\cU\to\cU_1$ is a smooth morphisms, with fibers projective spaces, hence it is a projective bundle. Moreover the line bundle $\cO_{\cU}(1)$, restriction of the tautological bundle $\cO_{\P(\Omega_X)}(1)$, has degree one on the fibers of $\pi_1$, hence $\cG:=\pi_{1*}\big(\cO_{\cU}(1)\big)$ is a rank $(2n-2r+1)$ vector bundle on $\cU_1$, whose Grothendieck projectivization is $\cU$.

Analogously, we may consider the projective subbundle $\HU\subset \cU\to\cU_1$; here the push-forward of the tautological line bundle $\cO_{\HU}(1):=\cO_{\cU}(1)_{|\HU}$ to $\cU_1$ provides a vector bundle $\HG$ of rank $2n-2r$, which is a quotient of $\cG$. We denote its kernel by $\cK$, so that we have an exact sequence of vector bundles on $\cU_1$: 
\begin{equation}\label{eq:Ctwo-up}
\shse{\cK}{\cG}{\HG}.
\end{equation} 
Note that, being the VMRT of $X$ at every point isomorphic to the variety $\cC$ (see Section \ref{ssec:statement}), it follows that $\cK$ has degree two on the fibers of $\cU_1\to X$.

\subsection{Restriction to rational curves}\label{ssec:resratcurves}

Let $\ell$ be a fiber of $p:\cU\to \cM$. We will study here the pullback to $\ell$ of the bundles constructed above, that is, the fiber products $\cU\times_X\ell$, ${\cU_i}\times_X\ell$, \dots, that we denote by $\cU_{|\ell}$, ${\cU_1}_{|\ell}$, etc. The main tool here is the fact that being $\ell$ a smooth rational curve, for every complex vector space $V$,  the natural map $H^1(\P^1,\GL(V))\to H^1(\P^1,\PGL(V))$ is surjective. 
Abusing notation, along this section we will denote again by $\theta$ the image of $\theta\in H^1(X,G)$ via the pullback map to $H^1(\ell,G)$.  
Fix a point $x\in q(\ell)$, 
and consider the 
commutative diagram:
$$
\xymatrix{H^1(\ell,\GL(\cF^\vee_{2,x})_{\cQ^\vee_x})\ar[r]\ar[d]&H^1(\ell,\GL(\cF^\vee_{2,x}))\ar@{->>}[d]\\H^1(\ell,\PGL(\cF^\vee_{2,x})_{\P(\cQ_x)})\ar[r]&H^1(\ell,\PGL(\cF^\vee_{2,x}))}
$$
The fact that the image of $\theta$ in $H^1(\ell,\PGL(\cF^\vee_{2,x}))$ can be lifted to $H^1(\ell,\GL(\cF^\vee_{2,x}))$, and that $\GL(\cF^\vee_{2,x})_{\cQ^\vee_x}\subset\GL(\cF^\vee_{2,x})$, 
implies that $\theta$ can be lifted to a cocycle $\theta'\in H^1(\ell,\GL(\cF^\vee_{2,x})_{\cQ^\vee_x})$. 

Via the natural homomorphisms $\GL(\cF^{\,\vee}_{2,x})_{\cQ^\vee_x}\to \GL(\cF^{\,\vee}_{2,x})$, $\GL(\cF_{2,x}^{\,\vee})_{\cQ^\vee_x}\to \GL(\cQ^\vee_x)$, $\GL(\cF^{\,\vee}_{2,x})_{\cQ^\vee_x}\to \GL(\cF^{\,\vee}_{1,x})$, the cocycle $\theta'$ defines three vector bundles over $\ell$, which we denote  by  $\cF_2^{\,\vee}$, $\cQ^{\,\vee}$, and $\cF_1^{\,\vee}$, respectively; the dual bundles fit into a short exact sequence: 
\begin{equation}\label{eq:Cone}
\shse{\cF_1}{\cF_2}{\cQ }.
\end{equation}

Note that the lifting $\theta'$ of $\theta$ is not unique; different choices of $\theta'$ provide different twists of the bundle $\cF_2$ and, subsequently, of the bundles $\cF_1$ and $\cQ$.

Now we consider also the action of $G$ on $\P(\Omega_{X,x})$. We may define a homomorphism of algebraic groups $\eta:\GL(\cF^{\,\vee}_{2,x})_{\cQ^\vee_x}\to \GL(T_{X,x})$ such that we have a commutative diagram:
$$
\xymatrix{\GL(\cF^\vee_{2,x})_{\cQ^\vee_x}\ar[r]^{\eta}\ar[d]&\GL(T_{X,x})\ar@{->}[d]\\G\,\ar@{>->}[r]&\PGL(T_{X,x})}
$$
In fact, an element $g$ of $\GL(\cF^\vee_{2,x})_{\cQ^\vee_x}$ is given (after choosing an appropriate basis of $T_{X,x}$) by a matrix of the form:
$$\left(\begin{array}{c|c}A_1&A_2\\\hline 0&A_3\end{array}\right),
$$
where $A_3$ defines the image of $g$ in $\GL(\cF^\vee_{1,x})$. Then it is enough to define $\eta(g)$ as given by the matrix:
$$
\left(\begin{array}{c|c}A_3\otimes A_1&A_3\otimes A_2\\\hline 0&A_3\otimes A_3\end{array}\right),
$$
where $\otimes$ denotes here the Kronecker product of matrices. Note that the map $\eta$ is not injective, but its kernel has order two.

Now, the image of $\theta'\in H^1(\ell,\GL(\cF^\vee_{2,x})_{\cQ^\vee_x})$ in $H^1(\ell,\GL(\Omega_{X,x}))$, given by the dual of the action of $\GL(\cF^\vee_{2,x})_{\cQ^\vee_x}$ on $T_{X,x}$, defines a vector bundle on $\ell$ whose projectivization is $\P({\Omega_{X}}_{|\ell})$, hence a bundle of the form ${\Omega_X}_{|\ell}\otimes\cO_\ell(d)$, for some $d\in\Z$. On the other hand, by construction, this bundle fits into the following 
commutative diagram, with short-exact rows and columns:
\begin{equation}
\xymatrix@=30pt{
\bigwedge^2\cF_1\ar[d]\ar@{=}[r]&\bigwedge^2\cF_1\ar[d]&\\
\bigotimes^2\cF_1\ar[r]\ar[d]&\cF_1\otimes \cF_2\ar[r]\ar[d]&\cF_1\otimes \cQ \ar@{=}[d]\\
S^2\cF_1\ar[r]&{\Omega_X}_{|\ell}\otimes\cO_\ell(d)\ar[r]&\cF_1\otimes \cQ \\
}\label{eq:diagram}
\end{equation}
Then, by choosing the lift $\theta'$ of $\theta$ appropriately, we may assume that $d=0,1$.

Let us denote by $\cL_i$ the tautological line bundles of ${\cU_i}_{|\ell}=\P(\cF_i)$, satisfying ${q_i}_*(\cL_i)=\cF_i$, for $i=1,2$, and by $\cL$ the restriction to $\cU_{|\ell}\subset\P({\Omega_X}_{|\ell}\otimes\cO_\ell(d))$ of the tautological line bundle whose push-forward to $\ell$ is ${\Omega_X}_{|\ell}\otimes\cO_\ell(d)$.
Then we may  write:
\begin{equation}\label{eq:Cblow}
\cL\cong \pi_1^*\cL_1\otimes \pi_2^*\cL_2.
\end{equation}
Moreover, we have the following relation, coming from the exact sequence (\ref{eq:Cone}):
$$
\pi_2^*\cL_2\cong\pi_1^*\cL_1\otimes\cH,
$$
where $\cH$ denotes the restriction to $\cU_{|\ell}$ of $\cO_\cU(\HU)$. In particular:
\begin{equation}\label{eq:Cexc}
\cH\cong\cL\otimes\pi^*_1\cL_1^{-2}.
\end{equation}
We finish by noting that, restricting to $\ell$ via the natural section $\ell \to \cU_{|\ell}$ induced by the inclusion $\ell \hookrightarrow \cU$, we may compute
\begin{eqnarray}
\cO_{\cU}(K_{\cU/X})\cdot \ell &=& \dim \cM_x =2n-r-1\label{eq:Cint3}\\
\label{eq:Cint1} \cL \cdot \ell &=& d-2\\ 
\label{eq:Cint2} \cH \cdot \ell &=& d-2 -2\pi_1^*\cL_1\cdot \ell
\end{eqnarray}

\section{The family of isotropic lines in $X$}\label{sec:isotrop}

The main goal of this section is to show that the variety $\HU$ is a $\P^1$-bundle over its image $\HM$ in $\cM$, which will then be the natural candidate for the family of isotropic lines in $X$. This will be done by showing that the intersection number of $\cO_{\cU}(\HU)$ with a fiber $\ell$ of $p$ is zero. 


\begin{lemma}\label{lem:c1s} With the same notation as in Section 3, 
$$\deg(\cF_2)  = -1 -(\pi_1^*\cL_1\cdot \ell+1)(2n-3r+2) + (2n-2r+1)d$$
\end{lemma}

\begin{proof}
 We write the relative canonical $\cO_{\cU}(K_{\cU/X})$ as 
\begin{eqnarray*}
 \cO_{\cU}(K_{\cU/X})& = &\cO_{\cU}(K_{\cU/\cU_2})  \otimes \pi_2^*\cO_{\cU_2}(K_{\cU_2/X})\\
 & = &\cO_{\cU}((r-1)\HU) \otimes \pi_2^*\cO_{\cU_2}(K_{\cU_2/X})
\end{eqnarray*} 
Now, restricting to $\cU_{|\ell}$, we get 
\begin{eqnarray*}
( \cO_{\cU}(K_{\cU/X}))_{|\ell}&=&\cH^{r-1}\otimes\pi_2^*\cL_2^{-2n+r}\otimes q^*\det(\cF_2)\\
& = &\cL^{-2n+2r-1}\otimes\pi^*_1\cL_1^{2n-3r+2} \otimes  q^*\det(\cF_2),
\end{eqnarray*} 
where the last equality follows from (\ref{eq:Cexc}).
Now we restrict to $\ell$ via the natural section $\ell \to \cU_{|\ell}$, and obtain the stated formula by (\ref{eq:Cint3}) and (\ref{eq:Cint1}).
\end{proof}

\begin{lemma}\label{lem:c1s2} With the same notation as in Section 3,
$$ \deg(\cF_1)  = -1+r(\pi_1^*\cL_1\cdot \ell+1)+\dfrac{r(r-1)}{2(2n-2r+1)}(\cH \cdot \ell)$$
\end{lemma}

\begin{proof} 
From the second column of diagram (\ref{eq:diagram}) we get
$$ \det(\Omega_X|_{\ell}) \otimes\cO_\ell( d\dim X)=\det(\cF_2)^{\otimes r} \otimes \det(\cF_1)^{\otimes 2n-2r+1},$$
which gives
$$(2n-2r+1) \deg(\cF_1)=(-2n+2r-1) +d\dim X- r\deg(\cF_2)-r;$$
substituting $\dim X= 2nr -(3r^2-r)/2$, with simple computations we get
$$ \deg(\cF_1)=-1 + d\dfrac{r(r-1)}{2(2n-2r+1)}+r(\pi_1^*\cL_1\cdot \ell+1)\left(1 + \dfrac{1-r}{2n-2r+1}\right)$$
which can be rewritten as
$$ \deg(\cF_1)=-1 +r(\pi_1^*\cL_1\cdot \ell+1)+ \dfrac{r(r-1)}{2(2n-2r+1)}\left(d-2 -2\pi_1^*\cL_1\cdot \ell\right)$$
%
%
and we conclude by formula (\ref{eq:Cint2}).
\end{proof}

\begin{lemma}\label{lem:familyU1}
Let $\ell$ be a fiber of $p$ and denote by $\ell_1$ its image in $\cU_1$. Then $\ell_1$ is a free rational curve, and so there exists a unique irreducible component $\cM_1$  of $\rat(\cU_1)$ containing the class of $\ell_1$. Moreover the natural morphism $\cM \to \cM_1$ is injective, and $-K_{\cU_1} \cdot \ell_1 \ge 2n -2r+2$, equality holding if and only if $\cM \to \cM_1$ is surjective.
\end{lemma}

\begin{proof}
The freeness of $\ell_1$ can be proved as in Remark \ref{rem:beautifulfree}, and it implies the smoothness of the scheme $\rat(\cU_1)$ at the point corresponding to $\ell_1$ (\cite[Theorem.~II.1.7, Theorem.~II.2.15]{kollar}). In particular, there is a unique component  containing $\ell_1$, that we denote by $\cM_1\subset\rat(\cU_1)$.
Since fibers of $p$ are mapped to different curves by $q$ (see Remark \ref{rem:beautifulfree}), and therefore by $\pi_1$, the natural morphism $\cM \to \cM_1$ is injective.
By \cite[Theorem. II.1.2]{kollar}) we have $\dim \cM_1 = -K_{\cU_1} \cdot \ell_1 + \dim \cU_1 -3$. 
Finally, being $\dim \cM= \dim \cU-1= \dim \cU_1 + 2n-2r -1$, the last assertion follows.
\end{proof}

\begin{corollary}\label{cor:linembed} With the same notation as above, we have
$$ \cO_{\cU}(\HU)\cdot \ell = d = \pi_1^*\cL_1 \cdot \ell +1= 0.$$
Then $\HM = p(\HU)$ parametrizes a subfamily of rational curves in $X$, which has codimension one in $\cM$, and the natural map $\cM \to \cM_1$ is an isomorphism. In particular  $\cM_1$ is a beautiful family of rational curves on $\cU_1$.\end{corollary}

%
\begin{proof} 
Since $-q^*K_{X} \cdot \ell=2n-r+1$, Lemma \ref{lem:familyU1} provides
\begin{equation}\label{eq:ineq}
1-r \le -\pi_1^*K_{\cU_1/X} \cdot \ell = -\pi_1^*K_{{\cU_1}_{|\ell}/\ell} \cdot \ell=r\pi_1^*\cL_1\cdot \ell-\deg(\cF_1),
\end{equation}
from which we get
$$\deg(\cF_1) \le r\big(\pi_1^*\cL_1\cdot \ell+1\big)-1.$$
Combining this with the expression of $\deg(\cF_1)$ obtained in Lemma \ref{lem:c1s} we get
$ \cO_\cU(\HU) \cdot \ell=\cH \cdot \ell \le 0$. On the other hand $\cO_\cU(\HU) \cdot \ell \ge 0$, since $\ell$ belongs to a dominating family and $\HU$ is an effective divisor, hence $\cO_\cU(\HU) \cdot \ell = 0$.
By formula (\ref{eq:Cint2}) we then have that $d$ is even, hence $d=0$, and $\pi_1^*\cL_1 \cdot \ell =-1$. 

Finally, by Lemma \ref{lem:c1s2}, these equalities provide $\deg(\cF_1)=-1$, and, together with (\ref{eq:ineq}), we get $-K_{\cU_1/X} \cdot \ell_1=1-r$. Then $-K_{\cU_1} \cdot \ell_1 = 2n -2r+2$, so,  by Lemma \ref{lem:familyU1} the morphism $\cM \to \cM_1$ is surjective. Since $\cM$ is smooth and $\cM_1$ is normal, then $\cM \simeq \cM_1$. In particular $\cM_1$ is unsplit; it is, moreover, beautiful, since its evaluation morphism $\pi_1$ is surjective and smooth.\end{proof}

\begin{remark}\label{rem:linVMRT}
Note that the VMRT of the family $\cM_1$ at every point $P$ 
of $\cU_1$ is a linear subspace of $\P(\Omega_{\cU_1,P})$. In fact  setting $x:=q_1(P)$ and $L:=q_1^{-1}(x)$, the tangent map $\tau_x:\cM_x\to \P(\Omega_{X,x})$ factors via $\tau_1:\cM_x\to\P({\Omega_{\cU_1}}_{|L})$: 
$$
\xymatrix@C=15mm{\cM_x\ar[r]_(.40){\tau_1}\ar@/^{7mm}/[rrr]^{\tau_x}\ar[dr]&\P({\Omega_{\cU_1}}_{|L})\ar@{-->}[r]\ar[d]&\P(\Omega_{X,x})\times L\ar[r]\ar[dl]&\P(\Omega_{X,x})\\&L&&}
$$
The tangent map of $\cM_1$ at $P$ is the restriction of $\tau_1$ to the fiber $\cM_{1,P}$ of $\cM_x\to L$ over $P$. Since $\cM_{1,P}$ goes one-to-one onto a linear subspace of $\P(\Omega_{X,x})$, and the second horizontal map is a linear projection (at every fiber over $L$), it follows that $\tau_1(\cM_{1,P})$ is also linear.
\end{remark}


\section{Reduction of the defining group}\label{sec:reduc}

As we have seen in Section \ref{ssec:VMRTCglobal}, the variety $\HU$ is a projective bundle $\P(\HG)$ over $\cU_1$. Given a point $x\in X$, the fiber of $\P(\HG)\to\cU_1$ over a point $y\in\pi_1^{-1}(x)$ is isomorphic to $\P(\cQ_x)$, so the vector bundle $\HG$ is defined by a  cocycle $\theta''\in H^1(\cU_1,\GL(\cQ_x))$. We will prove that there exists a skew-symmetric isomorphism $\omega_x:\cQ^\vee_x \to\cQ_x$, such that $\theta''\in H^1(\cU_1,\SP(\cQ^\vee_x))$, where $\SP(\cQ^\vee_x)$ denotes the subgroup  of $\GL(\cQ^\vee_x)$ preserving $\omega_x$. It is enough to prove that there exists a skew-symmetric isomorphism $\omega:\HG^\vee\to\HG \otimes \cB$, for some $\cB \in \Pic(\cU_1)$. We start by showing that $\cU_1$ admits a second contraction, which is a projective bundle. 

\subsection{Projective bundle structure of $\bm{\cU_1}$}

It is known (\cite[Theorem 1.1]{Ar}) that if a smooth complex projective variety $Y$ has a family of minimal rational curves whose VMRT at a general point is a $k$-dimensional linear subspace, then there is a dense open subset $Y^\circ \subset Y$ and a 
$\P^{k+1}$-fibration $\rho^\circ:Y^\circ \to W^\circ$ such that the general element of the family is a line in a fiber of $\rho^\circ$.
 The following Lemma shows that, under stronger assumptions, the fibration is defined on  the whole $Y$.

\begin{lemma}\label{lem:fibration}
Let $Y$ be a smooth variety, and let $p:\cV\to \cN$ be a beautiful family of rational curves, satisfying that the tangent map is regular on $\cV$, and the VMRT at every point $y \in Y$ is a linearly embedded $\P^k \subset \P(\Omega_{Y,y})$. Then there exists an equidimensional morphism $\rho:Y \to \cW$, which contracts curves parametrized by $\cN$, whose general fiber is a projective space of dimension $k+1$.
\end{lemma}

\begin{proof}
If the Picard number of $Y$ is one, then $Y$ is a projective space by \cite[Proposition 5]{Hw2} (or by a simple application of \cite[Theorem 1.1]{Ar}), so we can assume that $\rho_Y \ge 2$.
Let $q:\cV\to Y$ be the evaluation morphism associated with $\cN$. For a general $y \in Y$, since $\tau_y$ is the normalization of its image, we get that  $\cN_y \simeq \P^k$. Then all the fibers of $q$ are isomorphic to $\P^k$ by the smooth rigidity of the projective space (see \cite{Siu2}).

Given any point $y \in Y$, consider the $\P^1$-bundle $p_y:\cV_y\to\cN_y$ defined as the fiber product of $p:\cV\to\cN$ and $p_{|\cN_y}:\cN_y\to\cN$, whose image into $Y$ via the evaluation morphism is $\loc(\cN_y):=q(p^{-1}(p(\cN_y)))$; then   
$\loc(\cN_y)$ is an irreducible closed subset of $Y$ of dimension $k+1$.

We claim first that, given two points $y,z \in Y$ then $z \in \loc(\cN_y)$ if and only if $\loc(\cN_y)=\loc(\cN_z)$.

As in  \cite[Lemma 2.3]{AC} one can prove that $\cV_y$ is the blow-up of a $\P^{k+1}$ at a point, and that there is a generically injective map $g:(\P^{k+1},y') \to (\loc(\cN_y),y)$ which maps lines by $y'$ birationally to curves parametrized by $\cN_y$. 
Let $z'$ be a preimage  of $z$ via $g$ and let $\ell$ be the line passing by $y'$ and $z'$; then lines by $z'$  are mapped to curves in $Y$ passing by $z$ which are algebraically equivalent to $g(\ell)$.
Therefore these curves belong to a component of $\rat(Y)$ containing $g(\ell)$. Since this component is unique by the freeness of $g(\ell)$,  they belong to $\cN$, hence $\loc(\cN_y) \subset \loc(\cN_z)$. 
Since both the loci are irreducible and of the same dimension the claim follows.

Consider now the rationally connected fibration $\rho:\xymatrix{Y \ar@{-->}[r]& \cW}$ associated with $\cN$; by the claim every equivalence class has the same dimension, hence, by \cite[Proposition 1]{BCD} $\rho$ is an (equidimensional) morphism.

By \cite[Theorem II.1.2]{kollar} we have $\dim \cN=-K_Y \cdot g(\ell) + \dim Y - 3$, so $-K_Y \cdot g(\ell)=k+2$. By adjunction and by \cite{CMSB}, the general fiber is  a projective space of dimension $k+1$. 
\end{proof}

We now apply Lemma \ref{lem:fibration} to the case $Y=\cU_1$ and $\cV=\cM_1$, obtaining:

\begin{corollary}
There exist a smooth variety $\cW$ and a smooth morphism $\rho:\cU_1 \to \cW$, which contracts curves parametrized by $\cM_1$, whose fibers are projective spaces of dimension $2n-2r+1$.
\end{corollary}

\begin{proof} By Corollary \ref{cor:linembed} and Lemma \ref{lem:fibration}, there exists an equidimensional morphism $\rho:\cU_1 \to  \cW$, which contracts curves parametrized by $\cM_1$, whose general fiber is a projective space of dimension $2n-2r+1$. By \cite[Theorem 1.3]{HN}, $\rho$ is a projective bundle and $\cW$ is smooth.
\end{proof}

\subsection{Skew-symmetric form on $\bm{\HG}$}

We may now prove the following:

\begin{proposition}\label{prop:Uasrel} There exists a skew-symmetric isomorphism 
$$\omega:\HG^\vee\to\HG \otimes \cB,$$ 
for some $\cB \in \Pic(\cU_1)$.
\end{proposition}

\begin{proof}
By construction, since $\P(\cG)=\cU$ may be described as the universal family of lines in the fibers of $\rho:\cU_1\to \cW$, it follows that $\P(\cG)=\P(\Omega_{\cU_1/\cW})$. In particular there exists a line bundle $\cA_1$ on $\cU_1$ such that the corresponding tautological line bundles satisfy
$$\cO_{\P(\cG)}(1)=\cO_{\P(\Omega_{\cU_1/\cW})}(1)\otimes \pi_1^*\cA_1.$$
Since both $\cO_{\P(\cG)}(1)$ and $\cO_{\P(\Omega_{\cU_1/\cW})}(1)$ have degree $-2$ on a fiber of $p$, which is mapped to a line in a fiber of $\rho$,
there exists $\cA \in \Pic(\cW)$ such that $\cA_1= \rho^*\cA$, so $\cG=\Omega_{\cU_1/\cW} \otimes \rho^*\cA$.
In particular, the dual of the exact sequence (\ref{eq:Ctwo-up}) tensored with $\rho^*\cA$ is:
$$0 \to \HG^\vee \otimes \rho^*\cA \lra T_{\cU_1/\cW} \lra \cK^\vee\otimes \rho^*\cA \to 0.$$
But the sequence (\ref{eq:Ctwo-up}) tells us also that $\cO_\cU(\HU)= \cO_{\P(\cG)}(1) \otimes \pi_1^*\cK^\vee$, so, by Corollary \ref{cor:linembed}, $\pi_1^*\cK^\vee \cdot \ell =\cK^\vee \cdot \ell_1=2$.
Then the above sequence, restricted to a fiber $P\cong \P^{2n-2r+1}$ of $\rho$ becomes:
$$0 \to \HG^\vee_{|P} \to T_{P} \to \cO_{P}(2)\to 0.$$
The vector bundle $\cD:= \HG^\vee \otimes \rho^*\cA$ defines a distribution in $\cU_1$, contained in the relative tangent bundle $T_{\cU_1/\cW}$. Hence, we may consider the O'Neill skew-symmetric tensor (induced by the Lie bracket of holomorphic sections of $\cD$):
$$
\omega:\bigwedge^2\cD\to \dfrac{T_{\cU_1/\cW}}{\cD},
$$
providing a homomorphism of sheaves (that we denote also by $\omega$):
$$
\omega:\cD\to \cD^\vee\otimes\cK^\vee\otimes \rho^*\cA.
$$
In order to prove that $\omega$ is an isomorphism, we may check it on the restriction to every fiber $P\cong \P^{2n-2r+1}$ of $\rho$.  But on each of these fibers, $\cD_{|P}$ is the kernel of a surjective map $T_{P} \to \cO_{P}(2)$, hence it is a contact distribution (see \cite[Example~2.1]{LB}), and then we know that the O'Neill tensor  $\omega_{|P}:\cD_{|P}\to \cD^\vee_{|P}\otimes\cO_{P}(2)$ is an isomorphism. 
We conclude by setting $\cB= \cK^\vee \otimes \rho^*\cA^\vee$. 
\end{proof}

\section{Construction of  relative flags and conclusion}\label{sec:flags}

As we have already remarked, we will consider the family of rational curves parametrized by $\HM$, denoted $p:\HU\to \HM$, and the evaluation map $q:\HU\to X$, which is the fiber product $\cU_1\times_X\cZ$. We may consider, on one hand, the $\DA_{r-1}$-bundle $\ol{\cU}_1\to X$ associated to $\cU_1\to X$, whose fibers at every point $x\in X$ are the complete flags of $\P(\cF_{1,x})$. On the other hand, we may consider the fiber product:
$$
\xymatrix{\ol{\cU}_1\times_{\cU_1}{\HU}\ar[r]\ar[d]&\HU\ar[d]\\\ol{\cU}_1\ar[r]&\cU_1}
$$
Since  $\HU\to\cU_1$ is given by a cocycle with values in $\SP(\cQ^{\,\vee}_x)$, by Proposition \ref{prop:Uasrel}, the same holds for $\ol{\cU}_1\times_{\cU_1}{\HU}\to\ol{\cU}_1$, and we may construct its associated $\DC_{n-r}$-bundle,  $\ol{\HU}\to \ol{\cU}_1$, and the composition 
$$
\ol{q}:\ol{\HU}\lra \ol{\cU}_1\lra X
$$
is an $(\DA_{r-1}\sqcup\DC_{n-r})$-bundle over $X$, that factors through $q: \HU \to X$.


We number the nodes of the Dynkin diagram $\DA_{r-1}\sqcup\DC_{n-r}$ over the set $D=\{1,2,\dots,r-1,r+1,r+2,\dots,n\}$, as in  Figure (\ref{eq:dynkins2}) below, so that  the fibers of $q$ are the homogeneous manifolds corresponding to the marking of this diagram  at the set $I:=\{r-1,r+1\}$ (see for instance \cite[Section 2.2]{MOSWW}).

\begin{equation}\label{eq:dynkins2}
\begin{array}{c}\vspace{-0.1cm}
\ifx\du\undefined
  \newlength{\du}
\fi
\setlength{\du}{3.3\unitlength}
\begin{tikzpicture}
\pgftransformxscale{1.000000}
\pgftransformyscale{1.000000}

\definecolor{dialinecolor}{rgb}{0.000000, 0.000000, 0.000000} 
\pgfsetstrokecolor{dialinecolor}
\definecolor{dialinecolor}{rgb}{0.000000, 0.000000, 0.000000} 
\pgfsetfillcolor{dialinecolor}


\pgfsetlinewidth{0.300000\du}
\pgfsetdash{}{0pt}
\pgfsetdash{}{0pt}

\pgfpathellipse{\pgfpoint{-6\du}{0\du}}{\pgfpoint{1\du}{0\du}}{\pgfpoint{0\du}{1\du}}
\pgfusepath{stroke}
\node at (-6\du,0\du){};

\pgfpathellipse{\pgfpoint{4\du}{0\du}}{\pgfpoint{1\du}{0\du}}{\pgfpoint{0\du}{1\du}}
\pgfusepath{stroke}
\node at (4\du,0\du){};

\pgfpathellipse{\pgfpoint{14\du}{0\du}}{\pgfpoint{1\du}{0\du}}{\pgfpoint{0\du}{1\du}}
\pgfusepath{stroke}
\node at (14\du,0\du){};

\pgfpathellipse{\pgfpoint{24\du}{0\du}}{\pgfpoint{1\du}{0\du}}{\pgfpoint{0\du}{1\du}}
\pgfusepath{stroke}
\node at (24\du,0\du){};

\pgfpathellipse{\pgfpoint{34\du}{0\du}}{\pgfpoint{1\du}{0\du}}{\pgfpoint{0\du}{1\du}}
\pgfusepath{stroke}
\node at (34\du,0\du){};

\pgfpathellipse{\pgfpoint{44\du}{0\du}}{\pgfpoint{1\du}{0\du}}{\pgfpoint{0\du}{1\du}}
\pgfusepath{stroke}
\node at (44\du,0\du){};

\pgfpathellipse{\pgfpoint{54\du}{0\du}}{\pgfpoint{1\du}{0\du}}{\pgfpoint{0\du}{1\du}}
\pgfusepath{stroke}
\node at (54\du,0\du){};

\pgfpathellipse{\pgfpoint{64\du}{0\du}}{\pgfpoint{1\du}{0\du}}{\pgfpoint{0\du}{1\du}}
\pgfusepath{stroke}
\node at (64\du,0\du){};

\pgfsetlinewidth{0.300000\du}
\pgfsetdash{}{0pt}
\pgfsetdash{}{0pt}
\pgfsetbuttcap

{\draw (-5\du,0\du)--(3\du,0\du);}
{\draw (15\du,0\du)--(23\du,0\du);}
{\draw (35\du,0\du)--(43\du,0\du);}
{\draw (54.65\du,0.7\du)--(63.35\du,0.7\du);}
{\draw (54.65\du,-0.7\du)--(63.35\du,-0.7\du);}


{\pgfsetcornersarced{\pgfpoint{0.300000\du}{0.300000\du}}\definecolor{dialinecolor}{rgb}{0.000000, 0.000000, 0.000000}
\pgfsetstrokecolor{dialinecolor}
\draw (60.8\du,-1.2\du)--(57\du,0\du)--(60.8\du,1.2\du);}

\pgfsetlinewidth{0.400000\du}
\pgfsetdash{{1.000000\du}{1.000000\du}}{0\du}
\pgfsetdash{{1.000000\du}{1.00000\du}}{0\du}
\pgfsetbuttcap
{\draw (5.3\du,-1\du)--(13\du,-1\du);}
{\draw (45.3\du,-1\du)--(53\du,-1\du);}


\node[anchor=west] at (4.5\du,-4\du){${\rm A}_{r-1}$};
\node[anchor=west] at (44.5\du,-4\du){${\rm C}_{n-r}$};
\node[anchor=west] at (26.8\du,-6.5\du){$I$};
\node[anchor=west] at (21.2\du,-3.5\du){$\underbrace{\hspace{1.5cm}}$};

\node[anchor=south] at (-6\du,1.1\du){$\scriptstyle 1$};

\node[anchor=south] at (4\du,1.1\du){$\scriptstyle 2$};

\node[anchor=south] at (14\du,1.1\du){$\scriptstyle r-2$};

\node[anchor=south] at (24\du,1.1\du){$\scriptstyle r-1$};

\node[anchor=south] at (34\du,1.1\du){$\scriptstyle r+1$};

\node[anchor=south] at (44\du,1.1\du){$\scriptstyle r+2$};

\node[anchor=south] at (54\du,1.1\du){$\scriptstyle n-1$};

\node[anchor=south] at (64\du,1.1\du){$\scriptstyle n$};

\end{tikzpicture} \vspace{0.1cm}\\
\end{array}
\end{equation}

\begin{proposition}\label{prop:fine}
The variety $\ol{\HU}$ is the complete flag manifold of type $\DC_n$, and $X$ is a symplectic Grassmannian.
\end{proposition}
\begin{proof}

It is enough to prove that $\ol{\HU}$ is a complete flag manifold, since in this case, being a target of a contraction of $\ol{\HU}$, $X$ would be a rational homogeneous manifold; since the only rational homogeneous manifolds of Picard number one for which $\cM_x$ is as stated are the symplectic Grassmannians, the statement follows.

Now, in order to prove that $\ol{\HU}$ is a complete flag, we use \cite[Theorem A.1]{OSW}, for which we need to show that $\ol{\HU}$ admits $\rho_{\ol{\HU}}$ independent $\P^1$-bundle structures. This is obvious in the case $n=3$, $r=2$, so we may assume that $I\subsetneq D$.
Note that $\ol{\HU}$ has already $\rho_{\ol{\HU}}-1$ different $\P^1$-bundle structures, coming from the flag bundle construction described in Section \ref{ssec:flag}. Hence we only need to find the remaining $\P^1$-bundle structure, which will be constructed by means of minimal sections over isotropic lines.

Let $f:\P^1\to X$ be the normalization of any curve $\Gamma$ 
of the family $\HM$  and consider the pull-back $f^*\HU$ with the natural section $s:\P^1\to s(\P^1) \subset f^*\HU$.  Denote by $ \ol{s}: \P^1\to s^*{f}^*\ol{\HU}$ a minimal section of the bundle $s^*{f}^*\ol{\HU}$ over $\P^1$ and  by $\ol{\Gamma}$ its image.  We have then the following commutative diagram:

$$
\xymatrix@=35pt{
& s^*{f}^*\ol{\HU}\pb \ar@{^{(}->}[]+<0ex,-2.5ex>;[d]\ar[r]_{\pi} 
\ar@{^{(}->}[]+<0ex,-2.5ex>;[d]
& \,\,\,\P^1 \ar@{=}[rd]  \ar@{^{(}->}[]+<0ex,-2.5ex>;[d]^s \ar@/_{2mm}/[]+<-2ex,+2.2ex>;[l]+<+0.7ex,+2.2ex>_{\ol{s}} &\\
& f^*\ol{\HU}\pb\ar[d]^{\ol f}\ar[r]_{\pi}& f^*\HU\pb\ar[d]^{f}\ar[r]_q& \P^1\ar[d]^{f}\\
\HM & \ol{\HU}\ar[l]^{p\circ\pi}\ar[r]_{\pi}&\HU 
\ar@/^{3mm}/[]+<-2ex,-2ex>;[ll]+<+2ex,-2ex>_{p} \ar[r]_q&X }
$$

As in \cite[Section 4]{OSW} 
one can show that 
$\ol\Gamma$ is a minimal section of both the $(\DA_{r-1}\sqcup\DC_{n-r})$-bundle $f^*\ol{\HU}$ and the $(\DA_{r-2}\sqcup\DC_{n-r-1})$-bundle $s^*{f}^*\ol{\HU}$;  moreover  $s^*{f}^*\ol{\HU}$ is determined by the tagged Dynkin diagram obtained  by eliminating  the nodes indexed by $I$ from the tagged Dynkin diagram of  $f^*\ol{\HU}$.

Now we want to compute the tag $(d_1, \dots d_{r-1}, d_{r+1}, \dots d_n)$ of  $f^*\ol{\HU}$; to this aim we use  \cite[Formula (11)]{OSW}:
\begin{equation*}
\pi^*K_{\HU/X} \cdot \ol{\Gamma}=\sum_{i \in D} b_id_i -\sum_{j \not \in I}c_jd_j
\end{equation*}
where the $b_i$'s (resp. the $c_j$'s) can be read from \cite[Table 1]{OSW}, for $\cD=\DA_{r-1}\sqcup\DC_{n-r}$ (resp. $\cD=\DA_{r-2}\sqcup\DC_{n-r-1}$), obtaining
\begin{equation*}
2n-r-1= \pi^*K_{\HU/X} \cdot \ol{\Gamma}=(r-1)d_{r-1}+(2n-2r)d_{r+1}+\sum_{j \not \in I}a_jd_j,
\end{equation*}

$$a_j= \begin{cases} j & \textrm{for~} j=1, \dots, r-2\\
2n-2r & \textrm{for~} j=r+2, \dots, n-1\\
n-r & \textrm{for~} j= n
\end{cases}
$$
Now, since all the $d_k$ are nonnegative, and $d_{r-1},d_{r+1}$ are strictly positive (see \cite[Proof of (4.2.2)]{OSW}) we have that $d_{r-1}=d_{r+1}=1$ and $d_k=0$ for every other $k$.
In particular, by \cite[Proposition 3.17]{OSW}, the flag bundle $s^*f^*\ol{\HU}$ is trivial, and we may deduce, reasoning as in \cite[Corollary 4.3]{OSW}, that there exists a smooth projective variety $\ol{\HM}$ such that the morphism $\ol{\HU}\to \HM$ factors via a smooth $\P^1$-bundle $\ol{p}:\ol{\HU}\to \ol{\HM}$. Now, as in \cite[Proof of Theorem 1.1]{OSW} we can show that in $\NE(\ol{\HU})$ there exists $n=\rho_{\ol{\HU}}$ independent $K_X$-negative classes generating $n$ extremal rays, whose associated elementary contractions $\pi_i:X\to X_i$ are smooth $\P^1$-bundles, and we conclude by \cite[Theorem A.1]{OSW}.
\end{proof}

\bigskip

\noindent{\bf Acknowledgements.} The authors would like to thank J.A. Wi\'sniewski for the interesting discussions we had on this topic, during his visit to the University of Trento in 2016.



\begin{thebibliography}{10}

\bibitem{Ar}
Carolina Araujo.
\newblock Rational curves of minimal degree and characterizations of projective
  spaces.
\newblock {\em Math. Ann.}, 335(4):937--951, 2006.

\bibitem{AC}
Carolina Araujo and Ana-Maria Castravet.
\newblock Polarized minimal families of rational curves and higher {F}ano
  manifolds.
\newblock {\em Amer. J. Math.}, 134(1):87--107, 2012.

\bibitem{BPVV}
W.~Barth, C.~Peters, and A.~Van~de Ven.
\newblock {\em Compact complex surfaces}, volume~4 of {\em Ergebnisse der
  Mathematik und ihrer Grenzgebiete (3) [Results in Mathematics and Related
  Areas (3)]}.
\newblock Springer-Verlag, Berlin, 1984.

\bibitem{BCD}
Laurent Bonavero, Cinzia Casagrande, and St{\'e}phane Druel.
\newblock On covering and quasi-unsplit families of curves.
\newblock {\em J. Eur. Math. Soc. (JEMS)}, 9(1):45--57, 2007.

\bibitem{CMSB}
Koji Cho, Yoichi Miyaoka, and Nicholas~I. Shepherd-Barron.
\newblock Characterizations of projective space and applications to complex
  symplectic manifolds.
\newblock In {\em Higher dimensional birational geometry ({K}yoto, 1997)},
  volume~35 of {\em Adv. Stud. Pure Math.}, pages 1--88. Math. Soc. Japan,
  Tokyo, 2002.

\bibitem{De}
Olivier Debarre.
\newblock {\em Higher-dimensional algebraic geometry}.
\newblock Universitext. Springer-Verlag, New York, 2001.

\bibitem{HH}
Jaehyun Hong and Jun-Muk Hwang.
\newblock Characterization of the rational homogeneous space associated to a
  long simple root by its variety of minimal rational tangents.
\newblock In {\em Algebraic geometry in {E}ast {A}sia---{H}anoi 2005},
  volume~50 of {\em Adv. Stud. Pure Math.}, pages 217--236. Math. Soc. Japan,
  Tokyo, 2008.

\bibitem{HN}
Andreas H\"oring and Carla Novelli,
\newblock {\em Mori contractions of maximal length}.
\newblock {\em Publ. Res. Inst. Math. Sci.}, 49(1):215--228, 2013.   



\bibitem{Hw}
Jun-Muk Hwang.
\newblock Geometry of minimal rational curves on {F}ano manifolds.
\newblock In {\em School on {V}anishing {T}heorems and {E}ffective {R}esults in
  {A}lgebraic {G}eometry ({T}rieste, 2000)}, volume~6 of {\em ICTP Lect.
  Notes}, pages 335--393. Abdus Salam Int. Cent. Theoret. Phys., Trieste, 2001.

\bibitem{Hw2}
Jun-Muk Hwang.
\newblock On the degrees of {F}ano four-folds of {P}icard number 1.
\newblock {\em J. Reine Angew. Math.}, 556:225--235, 2003.

\bibitem{Hw15}
Jun-Muk Hwang.
\newblock {M}ori geometry meets {C}artan geometry: {V}arieties of minimal
  rational tangents.
\newblock In {\em Proceedings of the International Conference of
  Mathematicians, Seoul 2014}, volume~I, pages 369--394. Kyung Moon SA Co.
  Ltd., Seoul, 2014.

\bibitem{HM2}
Jun-Muk Hwang and Ngaiming Mok.
\newblock Birationality of the tangent map for minimal rational curves.
\newblock {\em Asian J. Math.}, 8(1):51--63, 2004.

\bibitem{Ke2}
Stefan Kebekus.
\newblock Families of singular rational curves.
\newblock {\em J. Algebraic Geom.}, 11(2):245--256, 2002.

\bibitem{KK}
Stefan Kebekus and S{\'a}ndor~J. Kov{\'a}cs.
\newblock Are rational curves determined by tangent vectors?
\newblock {\em Ann. Inst. Fourier (Grenoble)}, 54(1):53--79, 2004.

\bibitem{kollar}
J{\'a}nos Koll{\'a}r.
\newblock {\em Rational curves on algebraic varieties}, volume~32 of {\em
  Ergebnisse der Mathematik und ihrer Grenzgebiete. 3. Folge. A Series of
  Modern Surveys in Mathematics [Results in Mathematics and Related Areas. 3rd
  Series. A Series of Modern Surveys in Mathematics]}.
\newblock Springer-Verlag, Berlin, 1996.

\bibitem{LB}
Claude LeBrun.
\newblock Fano manifolds, contact structures, and quaternionic geometry.
\newblock {\em Internat. J. Math.}, 6(3):419--437, 1995.

\bibitem{Mk3}
Ngaiming Mok.
\newblock Recognizing certain rational homogeneous manifolds of {P}icard number
  1 from their varieties of minimal rational tangents.
\newblock In {\em Third {I}nternational {C}ongress of {C}hinese
  {M}athematicians. {P}art 1, 2}, volume~2 of {\em AMS/IP Stud. Adv. Math., 42,
  pt. 1}, pages 41--61. Amer. Math. Soc., Providence, RI, 2008.

\bibitem{Mk4}
Ngaiming Mok.
\newblock Geometric structures and substructures on uniruled projective
  manifolds.
\newblock In {\em Foliation Theory in Algebraic Geometry}, Simons Symposia,
  pages 103--148. Springer International Publishing, 2016.

\bibitem{MOSWW}
Roberto Mu{\~n}oz, Gianluca Occhetta, Luis~E. Sol{\'a}~Conde, Kiwamu Watanabe,
  and Jaros{\l}aw~A. Wi\'sniewski.
\newblock A survey on the {C}ampana-{P}eternell conjecture.
\newblock {\em Rend. Istit. Mat. Univ. Trieste}, 47:127--185, 2015.

\bibitem{OSWW}
Gianluca Occhetta, Luis~E. Sol{\'a}~Conde, Kiwamu Watanabe, and Jaros{\l}aw~A.
  Wi{\'s}niewski.
\newblock {F}ano manifolds whose elementary contractions are smooth $\mathbb
  {P}^1$-fibrations.
\newblock {\em Ann. sc. Norm. super. Pisa, cl. Sci. (5)}, 2017.
\newblock Preprint arXiv:{\tt 1407.3658}.

\bibitem{OSW}
Gianluca Occhetta, Luis~E. Sol{\'a}~Conde, and Jaros{\l}aw~A. Wi{\'s}niewski.
\newblock Flag bundles on {F}ano manifolds.
\newblock {\em J. Math. Pures Appl. (9)}, 106(4):651--669, 2016.



\bibitem{Siu2}
Yum~Tong Siu.
\newblock Errata: ``{N}ondeformability of the complex projective space'' [{J}.
  {R}eine {A}ngew.\ {M}ath.\ {\bf 399} (1989), 208--219; {MR}1004139
  (90h:32048)].
\newblock {\em J. Reine Angew. Math.}, 431:65--74, 1992.

\end{thebibliography}
\end{document}